\documentclass[12pt]{amsart}
\usepackage{amsmath}
\usepackage{amsfonts}
\usepackage{amssymb}
\usepackage{amsthm}
\usepackage[all]{xy}
\usepackage{color}

\newtheorem{theorem}{Theorem}[section]
\newtheorem{lemma}[theorem]{Lemma}
\newtheorem{proposition}[theorem]{Proposition}

\newtheorem{definition}[theorem]{Definition}

\numberwithin{equation}{section}

 \author[Nupur Patanker]{Nupur Patanker}
 \address{ Indian Institute of Science Education and Research, Bhopal}

\email{nupurp@iiserb.ac.in}

\author[Sanjay Kumar Singh]{Sanjay Kumar Singh}
\address{ Indian Institute of Science Education and Research, Bhopal}
\email{sanjayks@iiserb.ac.in}
\title{A short note on self-duality of Goppa codes on Elliptic and Hyperelliptic Function Fields}

\date{}
\begin{document}
\begin{abstract}
In this note, we investigate Goppa codes which are constructed by means of Elliptic function field and Hyperelliptic function field. We also give a simple criterion for self-duality of these codes.
\end{abstract}

\maketitle
\section{Introduction}
A linear code is a subspace of the $n$-dimensional standard vector space $\mathbb{F}_{q}^{n}$ over a finite field $\mathbb{F}_{q}$. Such codes are used for transmission of information. A linear code $C$ is called self-dual if $C=C^{\perp}$, where $C^{\perp}$ is the dual of $C$ with repect to Euclidean scalar product on $\mathbb{F}_{q}^{n}$. Self-dual codes are an important class of linear codes. \par
It was observed by Goppa in 1975 that we can use algebraic function fields over $\mathbb{F}_{q}$ to construct a class of linear codes by choosing a divisor and some rational places of algebraic function field over $\mathbb{F}_{q}$. In this note, we investigate codes which are constructed by means of elliptic and hyperelliptic function fields. This class of codes provide non-trivial examples of geometric Goppa codes. \par
 For self-dual geometric Goppa codes, Driencourt \cite{self2} and Stichtenoth \cite{self3} showed a criterion, which is too complex to apply. In \cite{self-dual}, Xing gave a simple criterion for self-duality of Goppa codes over elliptic function field with base field of characteristic 2. There are some difficulties in generalising the  Xing's criterion  to the field of characteristic not equal to $2$.  Using Xing's idea and results from \cite{equal}, we give a simple criterion for self-duality of Goppa codes over elliptic function field and hyperelliptic function field of characteristic not equal to $2$.  
\section{Preliminaries}
\subsection{Goppa code}
Goppa's construction  is described as follows:\\
Let $F / \mathbb{F}_{q}$ be an algebraic function field of genus $g$. Let $P_{1},\cdots,P_{n}$ be pairwise distinct places of $F/ \mathbb{F}_{q}$ of degree 1. Let $D=P_{1}+ \cdots +P_{n}$ and $G$ be a divisor of $F / \mathbb{F}_{q}$ such that $supp(G) \cap supp(D)=\emptyset$. The geometric Goppa code $C_{\mathcal{L}}(D,G)$ associated with $D$ and $G$ is defined by
$$C_{\mathcal{L}}(D,G):=\{(x(P_{1}), \cdots ,x(P_{n}))|~ x \in \mathcal{L}(G)\} \subseteq \mathbb{F}_{q}^{n}.$$\\
Then, $C_{\mathcal{L}}(D,G)$ is an $[n,k,d]$ code with parameters $k=dim(G)-dim(G-D)$ and $d \geq n-deg(G)$.\par
Another code can be associated with the divisors $G$ and $D$ by using local components of Weil differentials. We define the code $C_{\Omega}(D,G) \subseteq \mathbb{F}_{q}^{n}$ by
$$C_{\Omega}(D,G):=\{(\omega_{P_{1}}(1),\cdots,\omega_{P_{n}}(1))|~\omega  \in \Omega_{F}(G-D)\}.$$
Then, $C_{\Omega}(D,G)$ is an $[n,k',d']$ code with parameters $k'=i(G-D)-i(G)$ and $d' \geq deg(G)-(2g-2).$\par
The dual code of $C_{\mathcal{L}}(D,G)$ is $C_{\Omega}(D,G)$ i.e. $C_{\Omega}(D,G)=C_{\mathcal{L}}(D,G)^{\perp}.$ Let $\eta$ be a Weil differential such that $\nu_{P_{i}}(\eta)=-1$ and $\eta_{P_{i}}(1)=1$ for $i=1,\cdots,n$, then $C_{\mathcal{L}}(D,G)^{\perp}=C_{\Omega}(D,G)=C_{\mathcal{L}}(D,D-G+(\eta))$.

\subsection{Elliptic Function Field}
\begin{definition}
An algebraic function field $F/K$ (where $K$ is the full constant field of $F$) is said to be an elliptic function field if the following conditions hold:
\begin{itemize}
\item the genus of $F/K$ is $g=1$, and
\item there exists a divisor $A \in D_{F}$ with $deg~A=1$.
\end{itemize}
\end{definition}
The following theorems characterize elliptic function field over $K$ (where $char ~K \neq 2$).
\begin{theorem} [\cite{book}, Chapter VI]
Let $F/K$ be an elliptic function field. If $char~K \neq 2$, there exist $x,y \in F$ such that $F=K(x,y)$ and
$$y^{2}=f(x) \in K[x]$$ with a square-free polynomial $f(x) \in K[x]$ of degree 3.\\
\end{theorem}

\begin{theorem} [\cite{book}, Chapter VI] Suppose that $F=K(x,y)$ with 
$$y^{2}=f(x) \in K[x]$$
where $f(x)$ is a square-free polynomial of degree 3. Consider the decomposition $f(x)=c\prod_{i=1}^{r} p_{i}(x)$ of $f(x)$ into monic irreducible polynomials $p_{i}(x) \in K[x]$ with $0 \neq c \in K$. Denote by $P_{i} \in \mathbb{P}_{K(x)}$ the place of $K(x)$ corresponding to $p_{i}(x)$, and by $P_{\infty} \in \mathbb{P}_{K(x)}$ the pole of $x$. Then the following holds:
\begin{enumerate}
\item $K$ is the full constant field of $F$, and $F/K$ is an elliptic function field.
\item The extension $F/K(x)$ is cyclic of degree 2. The places $P_{1},\cdots,P_{r}$ and $P_{\infty}$ are ramified in $F/K(x)$; each of them has exactly one extension in $F$, say $Q_{1}, \cdots, Q_{r}$ and $Q_{\infty}$, and we have $e(Q_{j}|P_{j})=e(Q_{\infty}|P_{\infty})=2$, $deg~Q_{j}=deg~P_{j}$ and $deg ~Q_{\infty}=1$.
\item $P_{1},\cdots,P_{r}$ and $P_{\infty}$ are the only places of $K(x)$ which are ramified in $F/K(x)$, and the different of $F/K(x)$ is 
$$ Diff(F/K(x))=Q_{1}+\cdots+Q_{r}+Q_{\infty}.$$
\end{enumerate}
\end{theorem}

\subsection{Hyperelliptic function field}
\begin{definition}
A hyperelliptic function field over $K$ is an algebraic function field $F/K$ of genus $g \geq 2$ which contains a rational subfield $K(x) \subseteq F$ with $[F:K(x)]=2$.\\
\end{definition}

\begin{lemma} [\cite{book}, Chapter VI]
Assume that $char~K \neq 2$.
 \begin{enumerate}
\item Let $F/K$ be a hyperelliptic function field of genus $g$. Then there exist $x,y \in F$ such that $F=K(x,y)$ and 
\begin{equation}
y^{2}=f(x) \in K[x]
\end{equation}
with a square-free polynomial $f(x)$ of degree $2g+1$ or $2g+2$.
\item Conversely, if $F=K(x,y)$ and $y^{2}=f(x) \in K[x]$ with a square-free polynomial $f(x)$ of degree $m>4$, then $F/K$ is hyperelliptic of genus 
\begin{center}
$g=\left \{
\begin{array}{cc}
(m-1)/2 & \text{ if } m \equiv 1 ~~mod~~2,\\
(m-2)/2 & \text{ if } m \equiv 0 ~~mod~~2.\\
\end{array}
\right.$
\end{center}
\item Let $F=K(x,y)$ with $y^{2}=f(x)$ as in $(2.1)$. Then the place $P \in \mathbb{P}_{K(x)}$ which ramify in $F/K(x)$ are the following:

all zeros of $f(x)$ if $deg~f(x) \equiv 0 ~~mod~~2$,\\
all zeros of $f(x)$ and the pole of $x$ if $deg~f(x) \equiv 1 ~~mod~~2$.\\

\end{enumerate}
\end{lemma} 

\section{\textbf{Self-duality of Geometric Goppa codes over Elliptic Function Field $F/K$ with char $K \neq 2$}}
In \cite{self-dual}, Xing gave a criterion for self-duality of Goppa codes over elliptic function field with base field of characteristic $2.$ To get the criteria he has used the following proposition:
\begin{proposition}
Let $D=P_{1}+P_{2}+\cdots+P_{n}$, $G$ and $H$ be two positive divisors such that
\begin{enumerate}
\item $supp(D) \cap supp(G)=\emptyset$, $supp(D) \cap supp(H)=\emptyset$ and $supp(G) \cap supp(H)\neq \emptyset$.
\item $G^{\sigma}=G$ and $H^{\sigma}=H$ for any $\sigma \in Gal(\overline{\mathbb{F}_{q}}/\mathbb{F}_{q})$.
\item $deg(H)=2~deg(G)$.
\end{enumerate}
Then $C_{\mathcal{L}}(D,G)=C_{\mathcal{L}}(D,H-G)$ if and only if $H=2G$.

\end{proposition}

This proposition doesn't apply when $G$ and $H$ are not positive divisors (For example: Let $F/K$ be an algebraic function field, if $P_{\alpha} (\neq P_{\infty})$ be a place of $K(x)$ such that its extension in $F$ has degree $1$, then for $G=-P_{\infty}$ and $H=-P_{\infty}-P_{\alpha}$ we have $C_{\mathcal{L}}(D,G)=C_{\mathcal{L}}(D,H-G)$ but $H \neq 2G$). In this case, we observed that we can use following theorem to determine the self-duality of Goppa codes over elliptic and hyperelliptic function fields.\par

\begin{definition}
We call two divisors $G$ and $H$ equivalent with respect to $D$ if there exists $u \in F$ such that $H=G+(u)$ and $u(P_{i})=1$, for all $i=1,\cdots,n.$.
\end{definition}

\begin{theorem}[\cite{equal}, Corollary $4.15$]
Suppose $n>2g+2$. Let $G$ and $H$ be divisors of same degree $m$ on a curve of genus $g$. If $C_{\mathcal{L}}(D,G)$ is not equal to $0$ nor to $\mathbb{F}_{q}^{n}$ and $2g-1<m<n-1$, then $C_{\mathcal{L}}(D,G)=C_{\mathcal{L}}(D,H)$ if and only if $G$ and $H$ are equivalent with respect to $D$.
\end{theorem}

 We start with $K=\mathbb{F}_{q}$, $q$ large enough and characteristic of $K \neq 2$. Let $F/K$ be an elliptic function field. Then, $F=K(x,y)$ with $x,y \in F$ such that 
$$y^{2}=f(x) \in K[x]$$ with a square-free polynomial $f(x) \in K[x]$ of degree 3.\\\\
Consider the decomposition $f(x)=c\prod_{i=1}^{r} p_{i}(x)$ of $f(x)$ into monic irreducible polynomials $p_{i}(x) \in K[x]$ with $0 \neq c \in K$. Denote by $P_{i} \in \mathbb{P}_{K(x)}$ the place of $K(x)$ corresponding to $p_{i}(x)$, and by $P_{\infty} \in \mathbb{P}_{K(x)}$ the pole of $x$.\\\\
 Let $n>4$ an even positive integer.
Let $R_{1},\cdots, R_{\frac{n}{2}}$ be places of degree 1 of $K(x)$ such that
\begin{itemize}
\item $R_{i} \not \in \{P_{1},\cdots, P_{r}\} ~~(1 \leq i \leq \frac{n}{2})$
\item For each $i$, $R_{i}$ has exactly  two extensions in $F$ say, $S_{i,1}$ and $S_{i,2}$.
\end{itemize}
  Let $D=\sum_{i=1}^{\frac{n}{2}} S_{i,1}+S_{i,2}$. Let $g(x) \in K[x]$ such that $(g(x))_{0}=D$ in $F$. Let $G$ be a divisor of $F$ of degree $\frac{n}{2}$. Let $\eta$ be a differential in $F$ defined by $$\eta=\frac{g'(x) dx}{g(x)}.$$ 
Then, we get $\nu_{P}(\eta)=-1$ and $res_{P}(\eta)=1$ for all $P \in supp(D)$. Therefore, $C_{\mathcal{L}}(D,G)^{\perp}=C_{\mathcal{L}}(D,D+(\eta)-G)$. \\
Now, $D+(\eta)= D+(g'(x))+(dx)-(g(x))= D + (g'(x))-4Q_{\infty}+ Q_{1}+ \cdots Q_{r}+Q_{\infty}+nQ_{\infty}-D=(g'(x))+(n-3)Q_{\infty}+ Q_{1}+ \cdots Q_{r}$. Then the condition for self-duality of $C_{\mathcal{L}}(D,G)$ is given by the following theorem.

\begin{theorem}
With all conditions as above, $C_{\mathcal{L}}(D,G)$ is self-dual if and only if $(g'(x))=2G-(u)-(n-3)Q_{\infty}-Q_{1}-\cdots -Q_{r}$ for some $u \in F$ such that $u(P)=1$ for $P \in supp(D)$.
\end{theorem}
\begin{proof}
$C_{\mathcal{L}}(D,G)$ is self-dual iff $C_{\mathcal{L}}(D,G)= C_{\mathcal{L}}(D,G)^{\perp}=C_{\mathcal{L}}(D,D+(\eta)-G)$. \\
By  \cite{equal} corollary 4.15,
\begin{align*}
 & C_{\mathcal{L}}(D,G)=C_{\mathcal{L}}(D,D+(\eta)-G)\\
 \Leftrightarrow &~G=D+(\eta)-G+(u), \text{ for some } u \in F \text{ such that } u(P)=1 \text{ for each } P \in supp(D)\\
 \Leftrightarrow &~(g'(x))=2G-(u)-(n-3)Q_{\infty}-Q_{1}-\cdots -Q_{r}
 \end{align*}
\end{proof}

\section{\textbf{Self-duality of geometric Goppa codes over Hyperelliptic function field $F/K $ of genus 2 with char $K \neq 2$}}
Let $K=\mathbb{F}_{q}$, $q$ large enough and characteristic of $K \neq 2$. Let $F/K$ be a hyperelliptic function field of genus 2. Then there exist $x,y \in F$ such that $F=K(x,y)$ and 
$$y^{2}=f(x) \in K[x]$$
with a square-free polynomial $f(x)$ of degree 5. 
Then the places $P \in \mathbb{P}_{K(x)}$ corresponding to all zeroes of $f(x)$ and the pole of $x$ ramify in $F/K(x)$.\par
Let $n>6$ be an even positive integer. Let $P_{1},\cdots, P_{r}$ be zeros of $f(x)$ and $P_{\infty}$ pole of $x$. Therefore, $P_{1}, \cdots, P_{r},P_{\infty}$ ramify in $F/K(x)$. Let $R_{1},\cdots, R_{\frac{n}{2}}$ be places of degree 1 of $K(x)$ such that
\begin{itemize}
\item $R_{i} \not \in \{P_{1},\cdots, P_{r}\}, 1 \leq i \leq \frac{n}{2}$
\item For each $i$, $R_{i}$ has exactly  two extensions in $F$ say, $S_{i,1}$ and $S_{i,2}$.
\end{itemize}
 Let $D=\sum_{i=1}^{\frac{n}{2}} S_{i,1}+S_{i,2}$. Let $g(x) \in K[x]$ such that $(g(x))_{0}=D$ in $F$. Let $G$ be a divisor of $F$ of degree $\frac{n}{2}+1$. Let $\eta$ be a differential in $F$ defined by $$\eta=\frac{g'(x) dx}{g(x)}.$$ 
Then, we get $\nu_{P}(\eta)=-1$ and $res_{P}(\eta)=1$ for all $P \in supp(D)$. Therefore, $C_{\mathcal{L}}(D,G)^{\perp}=C_{\mathcal{L}}(D,D+(\eta)-G)$. \\ 
 Now, $D+(\eta)= D+(g'(x))+(dx)-(g(x))= D + (g'(x))-4Q_{\infty}+ Q_{1}+ \cdots Q_{r}+Q_{\infty}+nQ_{\infty}-D=(g'(x))+(n-3)Q_{\infty}+ Q_{1}+ \cdots Q_{r}$. Then the condition for self-duality of $C_{\mathcal{L}}(D,G)$ is given by the following theorem.

\begin{theorem}
With all conditions as above, $C_{\mathcal{L}}(D,G)$ is self-dual if and only if $(g'(x))=2G-(u)-(n-3)Q_{\infty}-Q_{1}-\cdots -Q_{r}$ for some $u \in F$ such that $u(P)=1$ for $P \in supp(D)$.
\end{theorem}
\begin{proof}
$C_{\mathcal{L}}(D,G)$ is self-dual iff $C_{\mathcal{L}}(D,G)= C_{\mathcal{L}}(D,G)^{\perp}=C_{\mathcal{L}}(D,D+(\eta)-G)$. \\By \cite{equal} corollary 4.15,
\begin{align*}
 & C_{\mathcal{L}}(D,G)=C_{\mathcal{L}}(D,D+(\eta)-G)\\
 \Leftrightarrow &~G=D+(\eta)-G+(u) \text{ for some } u \in F \text{ such that } u(P)=1 \text{ for each } P \in supp(D)\\
 \Leftrightarrow &~(g'(x))=2G-(u)-(n-3)Q_{\infty}-Q_{1}-\cdots -Q_{r}.
 \end{align*}
\end{proof}

\section{Concluding Remarks}
In this note, we have investigated Goppa codes over Elliptic and Hyperelliptic function fields with base field of characteristic not equal to $2$. We gave a simple criterion for self-duality of these codes.

\end{document}